\documentclass[11pt]{amsart}
\usepackage{amsmath}
\usepackage{amssymb}
\usepackage{amsfonts}
\usepackage{amsthm}
\usepackage{enumerate}
\usepackage[mathscr]{eucal}
\usepackage{graphicx}
\usepackage{color}

\newtheorem{theorem}{Theorem}[section]
\newtheorem{lemma}[theorem]{Lemma}
\newtheorem*{thma}{Theorem A}

\theoremstyle{definition}

\theoremstyle{remark}
\newtheorem{remark}[theorem]{Remark}

\numberwithin{equation}{section}

\def\R{{\mathbb R}}

\def\M{{\mathcal M}}

\def\intslash{\rlap{\kern  .32em $\mspace {.5mu}\backslash$ }\int}
\def\qsl{{\rlap{\kern  .32em $\mspace {.5mu}\backslash$ }\int_{Q_x}}}

\def\S{\mathbf S}

\def\emph#1{{\it #1 }}

\def\pari{\partial}

\def\eg{{\it e.g. }}

\def\lc{\lesssim}

\def\pv{\text{\rm p.v.}}
\def\alp{\alpha}

\def\eps{\varepsilon}

\def\tet{\theta}

\def\lam{\lambda}            \def\Lam{\Lambda}

\def\si{\sigma}

              \def\Om{\Omega}
\def\fr{\frac}

\newcommand{\Be}{\begin{equation}}
\newcommand{\Ee}{\end{equation}}
\newcommand{\Bes}{\begin{equation*}}
\newcommand{\Ees}{\end{equation*}}
\newcommand{\Bsp}{\begin{split}}
\newcommand{\Esp}{\end{split}}
\newcommand{\Bm}{\begin{multline}}
\newcommand{\Em}{\end{multline}}
\newcommand{\Bea}{\begin{eqnarray}}
\newcommand{\Eea}{\end{eqnarray}}
\newcommand{\Beas}{\begin{eqnarray*}}
\newcommand{\Eeas}{\end{eqnarray*}}
\newcommand{\Benu}{\begin{enumerate}}
\newcommand{\Eenu}{\end{enumerate}}
\newcommand{\Bi}{\begin{itemize}}
\newcommand{\Ei}{\end{itemize}}
\begin{document}

\title[Bilinear endpoint estimates for commutator]{Bilinear endpoint estimates for Calder\'on commutator with rough kernel}

\author{Xudong Lai}
\address{Xudong Lai: Institute for Advanced Study in Mathematics, Harbin Institute of Technology, Harbin, 150001, People's Republic of China
\endgraf
}
\email{xudonglai@hit.edu.cn\ xudonglai@mail.bnu.edu.cn}


\date{October 25, 2017}



\begin{abstract}
In this paper, we establish some bilinear endpoint estimates of  Calder\'on commutator $\mathcal{C}[\nabla A,f](x)$ with a homogeneous kernel when $\Om\in L\log^+L(\mathbf{S}^{d-1})$. More precisely, we prove that $\mathcal{C}[\nabla A,f]$ maps $L^q(\R^d)\times L^1(\R^d)$ to $L^{r,\infty}(\R^d)$ if $q>d$ which improves previous result essentially. If $q=d$, we show that Calder\'on commutator maps $L^{d,1}(\R^d)\times L^1(\R^d)$ to $L^{r,\infty}(\R^d)$ which is new even if the kernel is smooth. The novelty in the paper is that we prove a new endpoint estimate of the Mary Weiss maximal function which may have its own interest in the theory of singular integral.
\end{abstract}

\maketitle

\section{Introduction}
The purpose of this paper is to study some bilinear endpoint estimates
which are unsolved in the previous work of A. P. Calder\'on \cite{Cal65}, C. P. Calder\'on \cite{CCal75,CCal79}. Before stating our results, we give some notation and background.

In 1965, A. P. Calder\'on introduced the commutator defined by
\Be\label{e:11com}
\mathcal{C}[\nabla A, f](x)=\pv \int_{\R^d} \fr{\Om(x-y)}{|x-y|^d}\fr{A(x)-A(y)}{|x-y|}\cdot f(y)dy,
\Ee
which is called the {\it Calder\'on commutator\/}, here $\Om$ is a function defined on $\R^d\setminus\{0\}$ which satisfies:
\begin{equation}\label{e:11Omho}
\Om(r\tet)=\Om(\tet)\ \ \text{for all $r>0$ and $\tet\in\mathbf{S}^{d-1}$;}
\end{equation}
\begin{equation}\label{e:11Omcan}
\int_{\mathbf{S}^{d-1}}\Om(\tet)\tet^{\alp}d\si(\tet)=0\ \ \text{for multi-index $|\alp|=1$};
\end{equation}
and $\Om\in L^1(\mathbf{S}^{d-1})$. $\mathbf{S}^{d-1}$ is the unit surface in $\R^d$ and $d\si$ denotes the surface measure on $\mathbf{S}^{d-1}$.
It is easy to see that $\mathcal{C}[\nabla A,f](x)$ is well defined if $A$ and $f$ are smooth functions with compact supports.
Calder\'on commutator $\mathcal{C}[\nabla A,f](x)$ is a typical example of non-convolution Calder\'on-Zygmund singular integral.
We can regard $\mathcal{C}[\nabla A,f](x)$ as a generalization of
\Bes
\begin{split}
[A, S]f(x)&=A(x)S(f)(x)-S(Af)(x)\\
&=-\pv\fr{1}{\pi}\int_{\R}\fr{1}{x-y}\fr{A(x)-A(y)}{x-y}f(y)dy
\end{split}
\Ees
where $S=\fr{d}{dx}\circ H$ and $H$ denotes the Hilbert transform.
It is well known that the commutator $[A, S]$ is a fundamental operator in harmonic analysis and plays an important role in the theory of the Cauchy integral along Lipschitz curve in $\mathbb{C}$, the boundary value problem of elliptic equation on non-smooth domain, and the Kato square root problem on $\R$ (see \eg \cite{Cal65}, \cite{Cal78}, \cite{Fef74}, \cite{MC97}, \cite{Gra250} for the details). Recently, there has been a renewed interest into the commutator $[A,S]$ and {\it d-commutator} introduced by Christ-Journ\'e (see \cite{CJ87}) since it has application in the mixing flow problem (see \eg \cite{SSS15}, \cite{HSSS17}).

In this paper, we are interested in the following strong bilinear estimate (or weak type estimate)
\Be\label{e:11bc}
\|\mathcal{C}[\nabla A,f]\|_{L^r(\R^d)}\leq C\|\nabla A\|_{L^q(\R^d)}\|f\|_{L^p(\R^d)},
\Ee
with $\fr{1}{r}=\fr{1}{q}+\fr{1}{p}$. Let us recall some historic literature about the above inequality. We divide them into three cases (see also the complete picture of $(1/p,1/q)$ in Figure 1 in Theorem \ref{t:11}).

\emph{Case $r\geq1$.}A. P. Calder\'on \cite{Cal65} showed that if $\fr{1}{r}=\fr{1}{q}+\fr{1}{p}$ with $1<r<\infty$, $1<q\leq\infty$, $1<p<\infty$, then \eqref{e:11bc}
holds when $\Omega$ satisfies the condition \eqref{e:11Omho} and
either the property (p-i) or (p-ii) defined as follows:
\begin{enumerate}[(p-1).]
\item \qquad $\Om$ is even and $\Om\in L^1(\mathbf{S}^{d-1})$;

\item \qquad $\Om\in L\log^+L(\mathbf{S}^{d-1})$ is odd and satisfies \eqref{e:11Omcan}.
\end{enumerate}
Later C. P. Calder\'on \cite{CCal75} proved \eqref{e:11bc} is still true in the case $r=1$, $1<p<\infty$, $1=\fr{1}{q}+\fr{1}{p}$, and also the case $1<r=q<\infty$, $p=\infty$ where $\Omega$ satisfies the condition \eqref{e:11Omho} and either the property (p-i) or (p-ii).
For the endpoint case $(q,p)=(\infty,1)$, Y. Ding and the author \cite{DL15a} recently proved that if $\Om$ satisfies \eqref{e:11Omho}, \eqref{e:11Omcan} and $\Om\in L\log^+L(\mathbf{S}^{d-1})$, then $\mathcal{C}[\nabla A,f](x)$ is bounded from $L^\infty(\R^d)\times L^1(\R^d)$ to $L^{1,\infty}(\R^d)$, the weak $L^1$ space. The study of this topic in this case is quite related to weak (1,1) bound of rough singular integral (see \eg \cite{Chr88}, \cite{See96}).

\emph{Case $r<{d}/(d+1)$.}C. P. Calder\'on \cite{CCal79} gave an example shows that if $r<{d}/(d+1)$, $q\geq1$ and $p\geq1$, $\mathcal{C}[\nabla A,f](x)$ is unbounded on a ball in $\R^d$ for some functions $f,A$ satisfying $f\in L^p(\R^d)$ and $\nabla A\in L^q(\R^d)$.

\emph{Case ${d}/(d+1)\leq r<1$.}In the same paper \cite{CCal79}, C. P. Calder\'on proved that if $\fr{1}{r}=\fr{1}{q}+\fr{1}{p}$ with $\fr{d}{d+1}\leq r<1$, $1\leq q<d$, $1< p\leq\infty$, $\mathcal{C}[\nabla A,f]$ maps $L^q(\R^d)\times L^p(\R^d)$ to $L^{r,\infty}(\R^d)$ when $\Om$ satisfies the condition \eqref{e:11Omho} and either the property (p-i) or (p-ii). Specially in this case if $1<q<d$,  C. P. Calder\'on pointed out that $L^{r,\infty}(\R^d)$ space can be replaced by $L^{r}(\R^d)$ by using the interpolation theorem developed by himself in \cite{CCal75}.
If ${d}/(d+1)\leq r<1$, $q>d$, $p\geq1$, C. P. Calder\'on got the following result.
\begin{thma}[see Theorem D in \cite{CCal75}]
Suppose that $\Om$ satisfies \eqref{e:11Omho}, \eqref{e:11Omcan}, $\Om\in L^1(\mathbf{S}^{d-1})$ and the H\"ormander condition
\Be\label{e:11reg}
\int_{|x|\geq2|y|}\Big|\fr{\Om(x-y)}{|x-y|^d}-\fr{\Om(y)}{|y|^d}\Big|dy<+\infty.
\Ee
Then $\mathcal{C}[\nabla A,f](x)$ is bounded from $L^q(\R^d)\times L^1{(\R^d)}$ to $L^{r,\infty}(\R^d)$ where $\fr{1}{r}=\fr{1}{q}+1$ with $q>d$. Moreover we have $\mathcal{C}[\nabla A,f](x)$ is bounded from $L^q(\R^d)\times L^p(\R^d)$ to $L^r(\R^d)$ where $\fr{1}{r}=\fr{1}{q}+\fr{1}{p}$ with $q>d$ and $p>1$.
\end{thma}

Based on the previous theory of rough singular integral, now a natural question is that whether the conclusions in Theorem A hold if $\Om$ is a rough kernel. Also notice that there is a case $r=d/(d+1)$, $p=1$ and $q=d$ which is not developed even if the kernel satisfies the H\"ormander condition \eqref{e:11reg}. In this case is it possible to establish some kind of estimate like \eqref{e:11bc} or weak type estimate? Well, the present paper will give confirm answers to those questions. Our main results are as follows.
\begin{theorem}\label{t:11}
Let $\mathcal{C}[\cdot,\cdot]$ be defined in \eqref{e:11com}. Suppose $\Om$ satisfies $(\ref{e:11Omho}), (\ref{e:11Omcan})$
and $\Om\in L\log^+L(\mathbf{S}^{d-1})$ for $d\geq2$. Then we have the following conclusions:

{\rm(i).} For any $\lambda>0$, there exists a finite constant $C_{\Om,d}>0$ such that
\[ \lam^r|\{x\in \mathbb{R}^d:|\mathcal{C}[\nabla A,f](x)|>\lambda\}|\leq C_{\Om,d}\|{\nabla A}\|^r_{L^q(\R^d)}\|f\|^r_{L^1(\R^d)}, \]
where $\fr{1}{r}=\fr{1}{q}+1$ and $q>d$.

{\rm(ii).} Let $\fr{1}{r}=\fr{1}{d}+1$. Then for any $\lambda>0$, there exists a finite constant $C_{\Om,d}$ such that
\[ \lam^r|\{x\in \mathbb{R}^d:|\mathcal{C}[\nabla A,f](x)|>\lambda\}|\leq C_{\Om,d}\|{\nabla A}\|^r_{L^{d,1}(\R^d)}\|f\|^r_{L^1(\R^d)},\]
here $L^{d,1}(\R^d)$ is the standard Lorentz space (see \cite{Ste71}).

Combining these results of A. P. Calder\'on \cite{Cal65}, C. P. Calder\'on \cite{CCal75,CCal79}, Y. Ding and the author \cite{DL15a}, we may conclude all possible $({1}/{p},{1}/{q})$ in the following figure:
\begin{figure}[!h]
\centering
\includegraphics[height=0.38\textwidth]{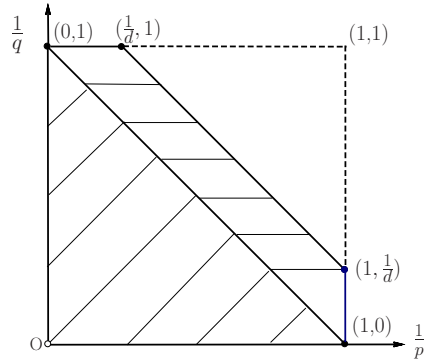}
\caption{\small{Our main results in Theorem \ref{t:11}  are corresponding to the case $0<\fr{1}{q}\leq \fr{1}{d}$ and $p=1$, see the blue line and point.}}
\end{figure}

\end{theorem}
\begin{remark}
Here we should point that the strong estimate in Theorem A (i.e $\mathcal{C}[\nabla A,f]$ maps $L^q(\R^d)\times L^p(\R^d)$ to $L^r(\R^d)$ where $\fr{1}{r}=\fr{1}{q}+\fr{1}{p}$ with $q>d$ and $p>1$) just follows from the multilinear interpolation theorem (see \cite{Gra250}, Theorem 7.2.2). The proof is standard. So we do not include this result in our main theorem.
\end{remark}
\begin{remark}
In (i) of Theorem \ref{t:11}, we only require $\Om\in L\log^+L(\mathbf{S}^{d-1})$ which is strictly weaker than the regularity condition \eqref{e:11reg}. Hence we improve Theorem A essentially. To the best knowledge of the author, the estimate in (ii) of Theorem \ref{t:11} is new even when the kernel is smooth.
\end{remark}

In \cite{CCal75}, C. P. Calder\'on used the method by making a Calder\'on-Zygmund decomposition of $f$ so that $f$ can be written as a good function and a bad function. It is easy to deal with the good function. However when estimating those terms related to the bad function, the H\"ormander condition \eqref{e:11reg} is crucial.
In this paper, the method here is different and it relies on our recent results \cite{DL15a} that $\mathcal{C}[\nabla A,f](x)$ is weak type (1,1) bounded if $A$ is a Lipschitz function and $\Om\in L\log^+L(\mathbf{S}^{d-1})$ (see Lemma \ref{l:11weak} below).
Roughly speaking, for a function $A$ satisfying $\nabla A\in L^q(\R^d)$, we will construct an {\it except set\/} which satisfies required weak type estimate. And on the complementary of {\it except set\/} the function $A$ is a Lipschitz function. Therefore the weak type (1,1) boundedness of $\mathcal{C}[\nabla A,f](x)$ could be applied there.

Notice that the estimate $\|\nabla A\|_{L^q(\R^d)}$ is related to the Sobolev space $W^{1,q}(\R^d)$. When $1\leq q<d$, Sobolev space $W^{1,q}(\R^d)$ is embedded into $L^{q*}(\R^d)$ with $q^*=dq/(d-q)$. This property is crucial in \cite{CCal79}. When $q>d$, {\it except set\/} can be constructed by using the Mary Weiss maximal operator $\M$ (see section \ref{s:112} for its definition), which maps $L^q(\R^d)$ to $L^q(\R^d)$ (or $L^{q,\infty}(\R^d)$) only when $q>d$. But when $q=d$,  Sobolev $W^{1,d}(\R^d)$ may be imbedded into an Orlicz space (see \cite{AF03}) which may be not useful to us.
This forces us to study the Mary Weiss maximal operator on $L^d(\R^d)$, which is quite difficult. Fortunately, we find a substitute that $\M$ maps the Lorentz space $L^{d,1}(\R^d)$ to $L^{d,\infty}(\R^d)$ which is enough to construct \emph{except set.}

Throughout this paper, we only consider the dimension $d\ge2$ and the letter $C$ stands for a positive finite constant which is independent of the essential variables and not necessarily the same one in each occurrence. $A\lc B$ means $A\leq CB$ for some constant $C$. Sometimes we write $C_\eps$ means that it depends on the parameter $\eps$. $A\approx B$ means that $A\lc B$ and $B\lc A$.
For a set $E\subset\R^d$, we denote by  $|E|$ or $m(E)$ the Lebesgue measure of $E$. $\nabla A$ will stand for the vector $(\pari_1A,\cdots,\pari_dA)$ where $\pari_i A(x)=\pari A(x)/\pari x_i$.
Define $$\|\nabla A\|_{X}=\Big\|\Big(\sum_{i=1}^d|\pari_iA|^2\Big)^{\fr{1}{2}}\Big\|_{X}$$
for $X=L^p(\R^d)$ or $X=L^{d,1}(\R^d)$.
\vskip1cm

\section {Proof of Theorem \ref{t:11}}\label{s:112}

\subsection{Some Lemmas}\quad
\vskip0.2cm
Before giving the proof of Theorem \ref{t:11}, we introduce some lemmas which play a key role in the proof of Theorem \ref{t:11}.
For those readers who are not familiar with the theory of the Lorentz space $L^{p,q}(\R^d)$, we refer to see \cite{Ste71}, Chapter V.3. We will use the theory of the Lorentz space $L^{p,q}(\R^d)$ in Lemma \ref{l:11md}.
Now we begin by some properties of a special maximal function which was introduced by Mary Weiss (see \cite{CCal75}). It is defined as
$$\M(\nabla A)(x)=\sup_{h\in\R^d\setminus\{0\}}\fr{|A(x+h)-A(x)|}{|h|}.$$
\begin{lemma}\label{l:mw}
Let $\nabla A\in L^p(\R^d)$ with $p>d$. Then $\M$ is bounded on $L^p(\R^d)$, that is
$$\|\M(\nabla A)\|_{L^p(\R^d)}\leq C\|\nabla A\|_{L^p(\R^d)},$$
where the constant $C$ is independent of $A$.
\end{lemma}
\begin{proof}
By using a standard limiting argument, we only need to consider $A$ as a $C^\infty$ function with compact support. Then the lemma just follows from the inequality
$$\fr{|A(x)-A(y)|}{|x-y|}\lc\Big(\fr{1}{|x-y|^d}\int_{|x-z|\leq 2|x-y|}|\nabla A(z)|^{q}dz\Big)^{\fr{1}{q}},$$
which holds for any $q>d$ (see Lemma 1.4 in \cite{CCal75}) and the fact that the Hardy Littlewood maximal operator is of strong $(p,p)$ for $p>1$.
\end{proof}
\begin{lemma}\label{l:11md}
Let $\nabla A\in L^{d,1}(\R^d)$. Then for any $\lam>0$, there exist a finite constant $C$ independent of $A$ such that
$$\lam^d|\{x\in\R^d:\M(\nabla A)(x)>\lam\}|\leq C\|\nabla A\|^d_{L^{d,1}(\R^d)}.$$
\end{lemma}
\begin{proof}
It suffices to consider $A$ as a smooth function with compact support. By the formula given in \cite{Ste70}, page 125, (17), we may write
\Bes
A(x)=C_n\sum_{i=1}^d\int_{\R^d}\fr{x_i-y_i}{|x-y|^d}\pari_i A(y)dy=K*f(x)
\Ees
where $K(x)=1/|x|^{d-1}$, $f=C_n\sum_{j=1}^dR_j(\pari_j A)$ with $R_j$ the Riesz transforms. By using the fact the Riesz transform $R_j$ maps $L^{d,1}(\R^d)$ to itself which follows from the general form of the Marcinkiewicz interpolation theorem (see \cite{Ste71}, Theorem 3.15 in page 197), one can easily get that
$$\|f\|_{L^{d,1}(\R^d)}\lc\|\nabla A\|_{L^{d,1}(\R^d)}.$$
Hence to prove the lemma, it is enough to show that
\Be\label{e:11weissmaximal}
\lam^d|\{x\in\R^d:\M(\nabla A)(x)>\lam\}|\lc\|f\|^d_{L^{d,1}(\R^d)}
\Ee
with $A=K*f$. In the following our goal is to prove that for any $x\in\R^d$, the estimate
$$|A(x+h)-A(x)|\lc |h|T(f)(x)$$
holds uniformly for $h\in\R^d\setminus\{0\}$ with $T$ an operator maps $L^{d,1}(\R^d)$ to $L^{d,\infty}(\R^d)$. Once we prove this, we get \eqref{e:11weissmaximal} and hence complete the proof of Lemma \ref{l:11md}. We write
\Bes
\begin{split}
A(&x+h)-A(x)\\
&=\int_{|x-y|\leq2|h|}|x+h-y|^{-d+1}f(y)dy-\int_{|x-y|\leq2|h|}|x-y|^{-d+1}f(y)dy\\
&\ \ \ \ +\int_{|x-y|>2|h|}\Big(|x+h-y|^{-d+1}-|x-y|^{-d+1}\Big)f(y)dy\\
&=I+II+III.
\end{split}
\Ees

Let us first consider $I$. By an elementary calculation, one may get $K\in L^{d',\infty}(\R^d)$ where $d'=d/(d-1)$. Set $B(x,r)=\{y\in\R^d:|x-y|\leq r\}$. Using the rearrangement inequality (see \cite{Gra249}, page 74, Exercise 1.4.1), we have
\Bes
\begin{split}
|I|&\leq\int_{\R^d}K(x+h-y)(f\chi_{B(x,2|h|)})(y)dy\leq\int_0^\infty K^*(s)(f\chi_{B(x,2|h|)})^*(s)ds\\
&\leq\Big(\int_0^\infty (f\chi_{B(x,2|h|)})^*(s)s^{\fr{1}{d}}\fr{ds}{s}\Big)\cdot\sup_{s>0}\Big(K^*(s)s^{\fr{1}{d'}}\Big)\\
&\lc\|f\chi_{B(x,2|h|)}\|_{L^{d,1}(\R^d)}\|K\|_{L^{d',\infty}(\R^d)},
\end{split}
\Ees
here $f^*$ represents the decreasing rearrangement of $f$.
By an elementary calculation, one may get $\|\chi_E\|_{L^{d,1}(\R^d)}=\|\chi_E\|_{L^d(\R^d)}$ holds for any characteristic function $\chi_E$ of set $E$ of finite Lebesgue measure, thus $\|\chi_{B(x,2|h|)}\|_{L^{d,1}(\R^d)}=C_d|h|$. Therefore we get
$$|I|\lc |h|\Lam(f)(x),\ \ \text{where}\ \ \Lam(f)(x)=\sup_{r>0}\fr{\|f\chi_{B(x,r)}\|_{L^{d,1}(\R^d)}}{\|\chi_{B(x,r)}\|_{L^{d,1}(\R^d)}}.$$
Below we need to show that the operator $\Lam$ maps $L^{d,1}(\R^d)$ to $L^{d,\infty}(\R^d)$. Since $L^{d,1}(\R^d)$ is a Banach space (see \cite{Ste71}, page 204, Theorem 3.22), it is sufficient to show that $\Lam$ maps the characteristic function $\chi_E\in L^{d,1}(\R^d)$ to $L^{d,\infty}(\R^d)$ (see \cite{Gra249}, page 62, Lemma 1.4.20). However
in this case, it is equivalent to show that
$$\lam|\{x\in\R^d: M(\chi_E)(x)>\lam\}|\lc\|\chi_E\|_{L^1(\R^d)},$$
where $M$ is the Hardy-Littlewood maximal operator. It is well known that $M$ is of weak type (1,1), hence we have shown that $\Lam$ maps $L^{d,1}(\R^d)$ to $L^{d,\infty}(\R^d)$.

Next we consider $II$. This estimate is quite simple. Since the kernel $k(x)=\eps^{-1}|x|^{-d+1}\chi_{\{|x|\leq \eps\}}$ is a radial non-increasing function and $L^1$ integrable in $\R^d$, we get
$$|II|\lc \|k\|_{L^1(\R^d)}|h|M(f)(x).$$
Notice that $L^{p,1}(\R^d)\subset L^p(\R^d)$ and $M$ is of strong type $(p,p)$, $1<p<\infty$, of course those imply that $M$ maps $L^{d,1}(\R^d)$ to $L^{d,\infty}(\R^d)$.

Finally we give an estimate of $III$. Notice that we only consider $|x-y|>2|h|$. Then by the Taylor expansion of $|x-y+h|^{-d+1}$, one may have
\Be\label{e:11taylor}
\fr{1}{|x-y+h|^{d-1}}-\fr{1}{|x-y|^{d-1}}=(-d+1)\sum_{j=1}^dh_j\fr{x_j-y_j}{|x-y|^{d+1}}+R(x,y,h)
\Ee
where the Taylor expansion's remainder term $R(x,y,h)$ satisfies
$$|R(x,y,h)|\leq C|h|^2|x-y|^{-d-1},\ \ |x-y|>2|h|.$$
Inserting \eqref{e:11taylor} into the term $III$ with the above estimate of $R(x,y,h)$, we conclude that
$$|III|\lc |h|\sum_{j=1}^d R_j^*(f)(x)+|h|^2\int_{|x-y|>2|h|}|x-y|^{-d-1}|f(y)|dy$$
where $R_j^*$ is the maximal Riesz transform which is defined by
$$R_j^*(f)(x)=\sup_{\eps>0}\Big|\int_{|x-y|>\eps}\fr{x_j-y_j}{|x-y|^{d+1}}f(y)dy\Big|.$$
Since $R_j^*$ is bounded on $L^p(\R^d)$, $1<p<\infty$, one immediately gets that $R_j^*$ maps $L^{d,1}(\R^d)$ to $L^{d,\infty}(\R^d)$.
The second term which controls $III$ can be dealt with the same as we do in the estimate of $II$ once we notice that the function $\eps|x|^{-d-1}\chi_{\{|x|>\eps\}}$ is radial non-increasing and $L^1$ integrable.
\end{proof}

\begin{lemma}[see Theorem 5.2 in \cite{DL15a}]\label{l:11weak}
Let $f\in L^1(\R^d)$ and $A$ be a Lipschitz function. Then for any $\lam>0$, we have
$$\lam |\{x\in\R^d:|\mathcal{C}[\nabla A,f](x)|>\lam\}|\leq C_{\Om,d}\|\nabla A\|_{L^\infty(\R^d)}\|f\|_{L^1(\R^d)}.$$
\end{lemma}
\vskip0.5cm
\subsection{Proof of (i) in Theorem \ref{t:11}}\quad
\vskip0.2cm
We start to prove (i) of Theorem \ref{t:11}. Let $\fr{1}{r}=\fr{1}{q}+1$ and $q>d$. By using a standard limiting argument, we only need to show that when $A$ and $f$ are $C^\infty$ functions with compact supports, the following inequality
\[ |\{x\in \mathbb{R}^d:|\mathcal{C}[\nabla A,f](x)|>\lambda\}|
\leq C_{\Om,d}\lam^{-r}\|{\nabla A}\|^r_{L^q(\R^d)}\|f\|^r_{L^1(\R^d)} \]
holds for any $\lam>0$ with the constant $C_{\Om,d}$ independent of $\lam$, $A$ and $f$.
By a simple scaling argument, we may assume that
$$\|f\|_{L^1(\R^d)}=\|\nabla A\|_{L^q(\R^d)}=\|\Om\|_{L\log^+L(\S^{d-1})}=1.$$

Now fix $\lam>0$. For convenience set $E_\lam=\{x\in\R^d:|\mathcal{C}[\nabla A,f](x)|>\lam\}$
and define the {\it except set}
\Bes
\begin{split}
J_{\lam}=\big\{x\in\R^d:\M (\nabla A)(x)>\lam^{\fr{r}{q}}\big\}.
\end{split}
\Ees
We need to show $|E_\lam|\lc\lam^{-r}.$ From Lemma \ref{l:mw}, $\M$ is bounded on $L^q(\R^d)$ with $q>d$. Hence $\M$ maps $L^q(\R^d)$ to $L^{q,\infty}(\R^d)$, i.e.
\Be\label{e:11jlam}
\begin{split}
|J_{\lam}|\lc\lam^{-r}\|\nabla A\|^{q}_{L^q(\R^d)}=\lam^{-r}.
\end{split}
\Ee

Choose an open set $G_\lam$ which satisfies the following conditions:
(1) $J_\lam\subset G_\lam$;
(2) $m(G_\lam)\leq2|J_\lam|$.
By the property \eqref{e:11jlam} of $J_\lam$, we see that $m(G_\lam)\lc\lam^{-r}$. Next making a Whitney decomposition of $G_\lam$ (see \cite{Gra249}), we may get a family of disjoint dyadic cubes $\{Q_k\}_k$ such that
\begin{enumerate}[(i).]
\item \quad $G_\lam=\bigcup_{k=1}^\infty Q_k$;
\item \quad $\sqrt{d}\cdot l(Q_k)\leq dist(Q_k,(G_\lam)^c)\leq4\sqrt{d}\cdot l(Q_k).$
\end{enumerate}
With those properties (i) and (ii), for each $Q_k$, we could construct a larger cube $Q_k^*$ so that $Q_k\subset Q_k^*$, $Q_k^*$ is centered at $y_k$ and $y_k\in (G_\lam)^c$, $|Q_k^*|\leq C|Q_k|$. The constant $C$ here is only dependent on the dimension. By the property (ii) above, the distance between $Q_k$ and $(G_\lam)^c$ equals to $Cl(Q_k)$.
Therefore by the construction of $Q_k^*$ and $y_k$, we get
\Be\label{e:11whitney}
dist(y_k,Q_k)\approx l(Q_k).
\Ee

Now we return to give an estimate of $E_\lam$. Split $f$ into two parts $f=f_1+f_2$ where $f_1(x)=f(x)\chi_{(G_\lam)^c}(x)$ and $f_2(x)=f(x)\chi_{G_\lam}(x)$. By the definition of $J_\lam$, when restricted on $(G_\lam)^c$, $A$ is a Lipschitz function with $\|\nabla A\|_{L^\infty((G_\lam)^c)}\leq\lam^{\fr{r}{q}}$. Let $\tilde{A}$ stand for the Lipschitz extension of $A$ from $(G_\lam)^c$ to $\R^d$ (see \cite{Ste70}, page 174, Theorem 3) so that
$$\tilde{A}(y)=A(y)\ \ \text{if} \ y\in (G_\lam)^c;$$
$$\big|\tilde{A}(x)-\tilde{A}(y)\big|\leq\lam^{\fr{r}{q}}|x-y|\ \ \text{for all}\ x,y\in\R^d.$$

Since the operator $\mathcal{C}[\cdot,\cdot]$ is bilinear, we split $E_\lam$ as three terms
\Bes
\begin{split}
|\{x\in\R^d: &|\mathcal{C}[\nabla A,f](x)|>\lambda\}|\\
&\leq |10G_\lam|+\big|\{x\in (10G_\lam)^c:|\mathcal{C}[\nabla A,f_1](x)|>\lambda/2\}\big|\\
&\ \ \ \ +\big|\{x\in (10G_\lam)^c:|\mathcal{C}[\nabla A,f_2](x)|>\lambda/2\}\big|.
\end{split}
\Ees
The first term above satisfies $|10G_\lam|\lc\lam^{-r}$, which is the required bound. In the following, we only consider $x\in(10G_\lam)^c$. By the definition of $f_1$, we see that $\mathcal{C}[\nabla A,f_1](x)=\mathcal{C}[\nabla\tilde{A},f_1](x)$.
With this equality in hand, Lemma \ref{l:11weak} implies
\Bes
\begin{split}
\big|\big\{x\in (10G_\lam)^c:& |\mathcal{C}[\nabla A,f_1](x)|>{\lam}/{2}\big\}\big|\\
&=\big|\big\{x\in (10G_\lam)^c: |\mathcal{C}[\nabla \tilde{A},f_1](x)|>{\lam}/{2}\big\}\big|\\
&\leq \lam^{-1}C_{\Om,d}\|\nabla\tilde{A}\|_{L^\infty(\R^d)}\|f_1\|_{L^1(\R^d)}\lc\lam^{-1+\fr{r}{q}}=\lam^{-r}.
\end{split}
\Ees

Let us turn to $\mathcal{C}[\nabla A,f_2](x)$, which can be rewritten as
\Be\label{e:11f2}
\mathcal{C}[\nabla A,f_2](x)=
\mathcal{C}[\nabla\tilde{A},f_2](x)+\int_{\R^d}\fr{\Om(x-y)}{|x-y|^d}\fr{\tilde{A}(y)-A(y)}{|x-y|}f_{2}(y)dy.
\Ee
Using the similar method of dealing with $\mathcal{C}[\nabla\tilde{A},f_1]$, we may get $$|\{x\in(10G_\lam)^c:|\mathcal{C}[\nabla\tilde{A},f_2](x)|>\lam/4\}|\lc\lam^{-r}.$$
Therefore it remains to consider the second term in \eqref{e:11f2}. Using the notation in the Whitney decomposition of $G_\lam$, we may write
\Bes
\begin{split}
\int_{\R^d}\fr{\Om(x-y)}{|x-y|^d}\fr{\tilde{A}(y)-A(y)}{|x-y|}f_{2}(y)dy&=\sum_k\int_{Q_k}\fr{\Om(x-y)}{|x-y|^d}\fr{\tilde{A}(y)-A(y)}{|x-y|}f(y)dy\\
&=:H(\tilde{A},f)(x)-H(A,f)(x),
\end{split}
\Ees
where
$$H(A,f)(x)=\sum_k\int_{Q_k}\fr{\Om(x-y)}{|x-y|^d}\fr{A(y)-A(y_k)}{|x-y|}f(y)dy.$$

By using the Chebyshev inequality, the Fubini theorem and $\tilde{A}$ is a Lipschitz function with Lipschitz bound $\lam^{r/q}$, we conclude that
\Bes
\begin{split}
|\{x\in(10G_\lam)^c&: |H(\tilde{A},f)(x)|>\lam/8\}|\\
&\lc\lam^{-1}\int_{(10G_\lam)^c}\sum_k\int_{Q_k}\fr{|\Om(x-y)|}{|x-y|^d}\fr{|\tilde{A}(y)-\tilde{A}(y_k)|}{|x-y|}|f(y)|dy\\
&\lc\lam^{-1+\fr{r}{q}}\sum_k\int_{Q_k}|y_k-y|\Big(\int_{(10G_\lam)^c}\fr{|\Om(x-y)|}{|x-y|^{d+1}}dx\Big)|f(y)|dy\\
&\lc\lam^{-r}\|\Om\|_{L^1(\S^{d-1})}\sum_k\int_{Q_k}|f(y)|dy\lc\lam^{-r},
\end{split}
\Ees
where in the second inequality we use the fact: $|x-y|\geq l(Q_k)\approx|y-y_k|$, since $x\in (10G_\lam)^c$ and \eqref{e:11whitney};
and in the last inequality we use $\|\Om\|_{L^1(\S^{d-1})}\leq\|\Om\|_{L\log^+L(\S^{d-1})}=1$.

Notice that by the construction of $y_k$, we have $y_k\in(G_\lam)^c$. It follows that $\mathcal{M}(\nabla A)(y_k)\leq\lam^{\fr{r}{q}}$. By using the Chebyshev inequality, the Fubini theorem and the above fact, we get
\Bes
\begin{split}
|\{x\in&(10G_\lam)^c:|H(A,f)(x)|>\lam\}|\\
&\lc\lam^{-1}\int_{(10G_\lam)^c}\sum_k\int_{Q_k}\fr{|\Om(x-y)|}{|x-y|^d}\fr{|A(y)-A(y_k)|}{|x-y|}|f(y)|dy\\
&\lc\lam^{-1}\sum_k\int_{Q_k}\mathcal{M}(\nabla A)(y_k)|y-y_k|\Big|\int_{(10G_\lam)^c}\fr{|\Om(x-y)|}{|x-y|^{d+1}}dx\Big||f(y)|dy\\
&\lc\lam^{-1+\fr{r}{q}}\|\Om\|_{L^1(\S^{d-1})}\sum_k\int_{Q_k}|f(y)|dy\lc\lam^{-r},
\end{split}
\Ees
where the third inequality follows from $|x-y|\geq l(Q_k)\approx|y-y_k|$. Therefore we complete the proof of (i) in Theorem \ref{t:11}.
$\hfill{} \Box$
\vskip0.5cm
\subsection{Proof of (ii) in Theorem \ref{t:11}}\quad
\vskip 0.2cm
The proof of (ii) is similar to that of (i) in Theorem \ref{t:11} once we choose $q=d$. The only difference is that when we give an estimate of {\it except set} $J_\lam$, we will use Lemma \ref{l:11md} instead of Lemma \ref{l:mw}. Proceeding the rest proof as we do in the proof of (i), we may obtain the result of (ii).
$\hfill{} \Box$
\vskip0.5cm
\subsection*{Acknowledgement} The author would like to thanks A. Seeger
for suggesting to consider the Lorentz space $L^{d,1}(\R^d)$ endpoint estimate of the Mary Weiss maximal operator $\M$.
\vskip1cm

\bibliographystyle{amsplain}

\end{document}